\documentclass[12pt, a4paper]{amsart}
\usepackage{amssymb,amscd, hyperref, epsf}

\begin{document}

\let\kappa=\varkappa
\let\epsilon=\varepsilon
\let\phi=\varphi
\let\p\partial

\def\Z{\mathbb Z}
\def\R{\mathbb R}
\def\C{\mathbb C}
\def\Q{\mathbb Q}
\def\P{\mathbb P}
\def\N{\mathbb N}
\def\L{\mathbb L}
\def\HH{\mathrm{H}}
\def\ss{X}

\def\conj{\overline}
\def\Beta{\mathrm{B}}
\def\const{\mathrm{const}}
\def\ov{\overline}
\def\wt{\widetilde}
\def\wh{\widehat}

\renewcommand{\Im}{\mathop{\mathrm{Im}}\nolimits}
\renewcommand{\Re}{\mathop{\mathrm{Re}}\nolimits}
\newcommand{\codim}{\mathop{\mathrm{codim}}\nolimits}
\newcommand{\id}{\mathop{\mathrm{id}}\nolimits}
\newcommand{\Aut}{\mathop{\mathrm{Aut}}\nolimits}
\newcommand{\lk}{\mathop{\mathrm{lk}}\nolimits}
\newcommand{\sign}{\mathop{\mathrm{sign}}\nolimits}
\newcommand{\pt}{\mathop{\mathrm{pt}}\nolimits}
\newcommand{\rk}{\mathop{\mathrm{rk}}\nolimits}
\newcommand{\SKY}{\mathop{\mathrm{SKY}}\nolimits}
\newcommand{\st}{\mathop{\mathrm{st}}\nolimits}
\def\Jet{{\mathcal J}}
\def\FC{{\mathrm{FCrit}}}
\def\sS{{\mathcal S}}
\def\lcan{\lambda_{\mathrm{can}}}
\def\ocan{\omega_{\mathrm{can}}}

\renewcommand{\mod}{\mathrel{\mathrm{mod}}}
\def\ds{\displaystyle}

\newtheorem{mainthm}{Theorem}
\newtheorem{thm}{Theorem}
\newtheorem{lem}[thm]{Lemma}
\newtheorem{prop}[thm]{Proposition}
\newtheorem{cor}[thm]{Corollary}

\theoremstyle{definition}
\newtheorem{exm}[thm]{Example}
\newtheorem{rem}[thm]{Remark}
\newtheorem{defin}[thm]{Definition}
\newtheorem{remark}[thm]{Remark}
\renewcommand{\thesubsection}{\arabic{subsection}}
\numberwithin{equation}{subsection}

\title[Types of connectedness of the constructive real number intervals] 
{Types of connectedness of the constructive real number intervals} 
\author[V.~Chernov]{Viktor Chernov}
\address{V.~Chernov, St Petersburg State University of Economics, Department of applied mathematics and economico-mathematical methods \\ 21 Sadovaya, St Petersburg 191023, Russia} 
\email{viktor\_chernov@mail.ru}


\begin{abstract} We study different notions of connected constructive metric spaces. They differ the types of connected components and how different components relate to each other. These notions are equivalent in classical point set topology but they give differ in the constructive world. In particular the interval of constructive real number appears to be connected if we use some of the definitions of a connected space and it is not connected when we use other definitions. 

This study is the continuation of the previous work of the author inspired by the question of Andrej Bauer about properties of locally constant functions. 
\end{abstract}


\maketitle

\section{Introduction}
Constructive Topology and Constructive Analysis study objects that can be computed using a certain algorithm, for example a Turing Machine. 
A Constructive Real Number (CRN) is an algorithmically given Cauchy sequence of rational numbers $\{x_n\}_{n=1}^{\infty}$ equipped with another algorithm that guarantees Cauchy convergence that is given any positive rational $\epsilon>0$ the algorithm constructs a nautral number $M$ such that for all $m,n>M$ we have $|x_m-x_n|<\epsilon.$ Two CRNs are equal if the members of the two Cauchy sequences with sufficiently large indexes are as close as desired. It is not hard to show that if two CRNs can not be not equal then they are equal. 
A Constructive Function (CF) is an algorithm that transforms CRNs to CRNs and of course it is required that equal CRNs are transformed to equal CRNs. All the number and functions in this paper are assumed to be constructive. 

A complete separable constructive metric space can be given by specifying an algorithmically enumerable set $P$ and a constructive metric on this set. The points of the space are algorithmically given Cauchy sequences of the elements of $P$ equipped with the algorithm that regulates convergence in itself. The metric can be naturally extended to this metric space and the set $P$ in this case appears to be everywhere dense. 
Constructive Mathematics was developed in the works of A. A. Markov \cite{Markov1, Markov2}, N.~A.~Shanin~\cite{Shanin} and their students and followers, see for example B.~A.~Kushner~\cite{Kushner}. A different but in many aspects similar approach to constructive mathematics was developed by E.~Bishop~\cite{Bishop} and his followers. 

In Constructive Mathematics a statement is not generally equivalent to its double negation. Markov's principle says that such an equivalence exists in the special situation: if an element of an algorithmically enumerable set can not not exist then it is possible to find such an element algorithmically. In other words every nonempty enumerable set is ``inhabitable''. In our proofs we will use Markov's principle and in this respect the allowed methods of the Russian Constructivism school are broader than those of Bishop's school.

Constructive versions of many classical results appear to be false. For example the constructive version of the intermediate value theorem and of the Brower fixed point theorem are wrong~\cite{Ceitin2},~\cite{Orevkov}. 
On the other side many statements that are false in the traditional classical mathematics appear to be true in the constructive world. For example every constructive function defined on the set of constructive real numbers is continuous~\cite{Ceitin1}. 

Our study of connectedness of an interval is based on constructive topological notions. An {\it open set\/} is a set of points equipped with an algorithm that given a point computes an interval with rational end points contained in this set. A set is {\it closed\/} if it is a compliment of an open set. Note that there are other very promising approaches to the constructivization of pointset toplogy. In the approach based on Abstract Stone Duality, see the works of Bauer and Taylor~\cite{BauerTaylor, Taylor},  and also Vickers~\cite{Vickers}, the topological structures are developed via a point free approach. This allows one to find constructive versions of the classical theorems. 

In this work we do not develop the general notions of constructive topological spaces and restrict ourselves to the study of the connectedness of an interval of constructive real numbers.

We will use the following abbreviations: CMS - Constructive Metric Space, $B(x,n)$ - an open ball of radius $2^{-n}$ centered at $x$.

\section{Definitions and Remarks}

\begin{defin} Let $X$ and $Y$ be two CMS and $F:X\to Y$ a constructive map. The map $F$ is {\it locally constant\/} if there is an algorithm $G$ transforming the points $x\in X$ to natural numbers such that all values $F(x)$ for $x\in B(x, G(x))$ are equal. The map $F$ is {\it potentially locally constant\/} if for every $x\in X$ there can not not be a natural number $n$ such that all the values of $F$ in the points of the ball $B(x,n)$ are equal. Clearly every locally constant function is potenatially locally constant.
\end{defin}

Theorem~\ref{theorem1} that generalizes Theorem 1 of our work~\cite{Chernov} shows that not only locally constant but even potentially locally constant functions given on an interval of CRNs are in fact constant functions. 
	
\begin{defin} Let $A$ be a set of points of CMS $X$. We say that this set $A$ is {\it open} in $X$ if one can specify an algorithm $G:A\to \N$ that transforms a point $a\in A$ to a natural number $G(a)$ such that $B(a, G(a))\subset A.$

A set $A$ is {\it pseudo open} in $X$ if for every point $a\in A$ there can not not exist a natural number $n$ such that $B(a,n)\subset A.$

A set $A$ is {\it closed} in $X$ if it is a compliment of an open set. 

A set $A$ is {\it sequentially closed\/} in $X$ if every algorithmically given sequence of points of $A$ that converges constructively to some limit has this limit in $A.$
\end{defin}

\begin{remark}
In classical point set Topology the notions of open and pseudo open subsets are equivalent, similarly for metric spaces the notions of closed and sequentially closed subsets are also equivalent. In the Constructive Topology this is not so even in the case of a constructive real line. Let us not a few simple relations between these notions. 
\begin{itemize}
\item A compliment to any open set is closed, the compliment to any closed set is pseudo open;
\item Every open subset is pseudo open, but the converse statement fails. The set $B$ from Theorem~\ref{theorem5} is pseudo open, but it is not open by Theorem~\ref{theorem3};
\item Every closed set is sequentially closed, but the converse statement is false. Set $A$ of Theorem~\ref{theorem5} is sequentially closed but it is not closed by Theorem~\ref{theorem4}.
\end{itemize}
\end{remark}

\begin{defin}
A subset of a CMS is {\it suitable} if it together with every point of the CMS contains all points equal to it~\cite{Chernov}. 
Consider two nonempty nonintersecting suitable subsets of CMS $X.$ Let us formulate for them the three possible conditions of being complimentary subsets. 

We call $X$ {\it hard separated} into these two subsets if every point of $X$ belongs to their union. 

We say that $X$ is {\it weakly separated\/} into these two subsets if the statement that a point of $X$ does not belong to one of these sets implies that it belongs to the other subset. 

We say that $X$ is {\it softly separated\/} into these two subsets if every point of $X$ can not not belong to their union.
\end{defin}

\begin{remark}
In classical mathematics these three types of separation are equivalent. If these two subsets are open (or closed) then in each of the three types of the space separation the total space is not connected. The Constructive version of this observation fails and the three types of separation yield different consequences. 
The hard separation assumes existence of the algorithm that given a point of $X$ tells which of the two subsets the point belongs to. The soft and the weak separations do not assume the existence of such an algorithm. 

\begin{itemize}
\item The property of being soft complimentary can be formulated as a double negation of both hard and soft complimentary. Namelly, if two nonempty and non intersecting subsets of a space softly separate it, then the statement that a point does not belong to one of the two subsets implies that it can not not belong to the other subset. 
\item Hard separation of a space implies weak separation that in turn implies soft separation. The converse statements are generally false. Soft and weak separations do not have to hard ones. Theorem~\ref{theorem2} the constructive interval $[a,b]$ can not be hard separated into two suitable subsets. However one can weakly (and hence softly) separate it into a closed $[a,c]$ and a half open $(c,b]$. 

The soft separation does not have to be a weak separation as one sees from the following example. We call a CRN {\it quasi rational} if it equals to some rational number. Let $A$ be the set of quasi rational numbers in an interval and $B$ be the complimentary set of CRNs in this interval. The interval is  softly separated into $A$ and $B$ but it is not weakly separated into these two subsets.  
\item In~\cite[Theorem 3]{Chernov} of the author we formulated the statement that it is impossible to separate the constructive interval into sequentially closed subsets. 
This statement appears to be wrong. In the current work we show that it is impossible to do a hard separation of the constructive interval into any subsets (Theorem~\ref{theorem2}). We consider four types of subsets: open, closed, pseudo open and sequentially closed. 

In Theorems~\ref{theorem3},~\ref{theorem4} we show that a constructive interval can not be softly (or weakly) separated into two open or two closed subsets. However it is possible to do a weak and soft separation of the interval into pairs of subsets of other types. In the proof of Theorem~\ref{theorem5} the interval is weakly and softly separated into an open and a pseudo open subset (and hence into two pseudo open ones) and it is separated into a closed and a sequentially closed subset (and hence into two sequentially closed subsets).
\end{itemize}
\end{remark}

\section{Theorems and Lemmas}

We will use the following 

\begin{lem}[\cite{Kushner}]\label{lemma1}
If two CRS can not be not equal then they are equal.
\end{lem}	

\begin{thm}\label{theorem1}
Every constructive potentially locally constant function $f$ with constructive input and output variables given on an interval is a constant function on this interval.
\end{thm}

\begin{proof}
Let us assume that the function $f$ is not constant. Then there can not not exist two points $p,q$ of the interval such that $f(p)\neq f(q).$ Without the loss of generality 
$p<q$ and we put $r=(p+q)/2.$ We compute $f(r)$ with the precision high enough that would allow us to determine which one of the statements $f(r)\neq f(q)$ and $f(r)\neq f(p)$ is true. We take one of the two halves $[p,r]$ and $[r,q]$ of the interval $[p,q]$ for which the function takes different values on its ends. We continue the construction in a similar fashion to get a contracting sequences of nested intervals. The values of $f$ on the left and the right end points of the intervals are different. However the sequence of the left end points and of the right end points gives the same constructive number $d.$ Since the function $f$ is potentially locally constant there can not not exist an open neighborhood of $d$ where the values of $f$ are not all equal to $f(d).$ Thus there can not not exist a natural number $N$ such that after it all the members of the two sequences of the values of $f$ (on the left and right interval end points) are equal. Thus we get a contradiction. Function $f$ can not be non constant, its values at different points can not not be equal. By Lemma~\ref{lemma1} they are equal.
\end{proof} 

\begin{thm}\label{theorem2}
An interval $I$ of $CRNs$ can not be hard separated into any suitable subsets.
\end{thm}
	
\begin{proof} Let us assume that the interval is hard separated into two suitable subsets $A$ and $B.$ This means that it is possible to construct an algorithm $G$ transforming the points of the interval into $\{0,1\}$ and $G(A)=0, G(B)=1.$ The algorithm $G$ is a constructive function thus it is continuous by the result of Ceitin~\cite{Ceitin1}. Thus it is locally constant. Theorem~\ref{theorem1} says that this function is in fact constant. Hence one of the subsets $A$ and $B$ is empty which contradicts to the definition of separation. 
\end{proof}

\begin{lem}\label{lemma2}
Let $X$ be a complete separable CMS, $D$ be its enumerable everywhere dense subset, $A$ be a suitable enumerable subset of points $X.$ Then provided that $A\neq \emptyset$ one can construct a point of $D\cap A.$
\end{lem}

\begin{proof} Let $P$ be an algorithm that transforms natural numbers to natural numbers and that has an undecidable domain $E,$ see~\cite{Ceitin1, Kleene, Kushner}. 
We consider the step by step process of applying $P$ to a number $n.$ We put $Q(n, k) = k,$  if algorithm $P$ did not yet finish working on an integer $n$ by step $k$, and otherwise we put $Q(n, k) = m,$ where $m$ is the step number on which $P$ produced the result on input $n.$ 

Let $a\in A$ and $\{d_k\}_{k=1}^{\infty}$ be a sequence of elements of $D$ converging to $a.$ For each natural $n$ we define a sequence
$\{x_{n_k}\}_{k=1}^{\infty}$ of the points of the space $X$ via the formula $x_{n_k}= d_{Q(n,k)}.$ If $n\not \in E$ then the sequence $x_{n_k}$ converges to $a.$
If $n\in E$ then the sequence $x_{n_k}$ stabilizes after a certain $k$ on the corresponding members of the sequence $\{d_k\}_{k=1}^{\infty}.$ 

In both cases the sequence $x_{n_k}$ converges. If $n\not \in E$ then the limit belongs to the enumerable subset $A.$ Thus there exists $n\in E$ for which the limit i.e. the element of $D$ belongs to $A$, by the so called capture principle~\cite{Ceitin1, Kushner}. 

Thus given $a\in A$ we can find $d\in D\cap A.$ If $a\in A$ can not not exist then this $d$ also can not not exist. 
 $C\cap A$ is enumerable so the Markov's principle allows us to find $d\in D\cap A.$
\end{proof}

\begin{thm}\label{theorem3}
An interval $I$ of CRNs can not be softly separated into two open subsets. 
\end{thm}

\begin{proof}
Let us assume that $A,B$ are two nonempty nonintersecting open subsets that softly separate the interval $I.$ We present each of the two open sets as the union of an enumerable collection of intervals with rational ends (this is possible to do by the Constructive version of the Lindeloef Theorem~\cite{Kushner}). 

The set of all CRNs that belongs to an open interval is enumerable, thus the set of all  CRNs that belongs to an enumerable union of open intervals is also enumerable. Thus the sets $A, B$ are suitable non empty and non intersecting enumerable subsets of the interval $I.$

We consider sets of rational numbers $A_0\subset A$ and $B_0\subset B.$

Using Lemma~\ref{lemma2} we can find an element in each of these sets $A_0$ and $B_0.$ Now we take an interval with end points in $A_0$ and $B_0.$ We divide it into two equal halves and take the one of the halves for which the ends belong to the different subset $A_0$ and $B_0.$
We continue this process and we get a contracting sequences of nested intervals the end points of which belong to two different sets $A_0$ and $B_0.$ These two sequences define equal CRNs. Since the sets $A$ and $B$ are suitable, both CRNs belong to one of them. These sets are open, thus together with the CRN the open set contains a tail of each of these two sequences. 

We get a contradiction and our assumption that one can softly separate the interval into two open subsets is false.
\end{proof}

\begin{cor}
An interval $I$ of CNRs can not be weakly separated into two open subsets.
\end{cor}

\begin{thm}\label{theorem4}
An interval $I$ of CRNs $I$ can not be softly separated into two closed subsets.
\end{thm}

\begin{proof} We assume that the subsets $V,W$ are closed non empty non intersecting that softly separate the constructive interval $I.$ Each of these sets is a compliment of an open set that we denote by $A$ and $B$ respectively. Let us prove that $A$ and $B$ softly separate $I$ which would contradict to the previous Theorem.

First of all $A, B$ are both non empty. Indeed if $A$ is empty then $V=I$ ad then $W$ is empty.

Second of all $A$ and $B$ do not intersect. Indeed if $x\in A\cap B$, then $x\not \in V$ and at the same time $x\not \in W$ and the sets $V,W$ would not satisfy the condition of being softly complimentary.

Third of all $A$ and $B$ softly compliment each other: if $x\not \in A$, then $x$ can not not belong to $B$; similarly if $x\not \in B$, then $x$ can not not belong to $A.$ Indeed assume that $x\not \in A,$ then $x\in V$. Since $V\cap W=\emptyset$ we get that $x\not \in W.$ Thus $x$ can not not belong to $B.$ 

Thus if the constructive interval is softly separated into two closed sets then it is softly separated into two open sets. 
\end{proof}

\begin{cor} An interval $I$ of CRNs can not be weakly separated into two closed subsets.
\end{cor}

\begin{thm}\label{theorem5}
Let $\{s_n\}_{n=1}^{\infty}$ be a strictly increasing algorithmical sequence of rational numbers from the interval $[a,b]$ that does not have a constructive limit
(Specker Sequence,~\cite{Specker}). Consider two sequences of intervals $\{A_n\}_{n=1}^{\infty}$ and $\{B_n\}_{n=1}^{\infty}$ where $A_n=[a, s_n)$ and $B_n=[s_n,b].$ Put $A=\cup_n A_n$ and $B=\cap_n B_n.$

Then the interval $[a,b]$ is weakly separated by $A$ and $B.$ The set $A$ is open, pseudo open and sequentially closed. The set $B$ is closed, pseudo open and sequentially closed.
\end{thm}

\begin{proof}

\begin{itemize}
\item By our construction the sets $A,B$ are non empty and do not intersect. Let us check that they satisfy the condition of being complimentary. Let $x\not \in A$, then $x\not \in A_n$ for all $n.$ Then $x\in B_n$ for all $n$ and hence $x\in B.$

Assume that $x\not \in B$, then there can not not exist $n$ such that $x\not \in B_n,$ this $x<s_n.$ The last relation is enumerable so by the Markov principle we can find such $n.$ Then $x$ is in the corresponding $A_n$ and hence $x\in A.$

\item By our construction the set $A$ is open and hence pseudo open. Let us show that it is sequentially closed at the same time. Consider the converging sequence $\{a_n\}_{n=1}^{\infty}$ of the points of $A$ with limit $a.$ Let us show that $a\in A.$ We observe that for every $n$ there is $m$ such that $a_n\in A_m,$ and thus $a_n<s_m.$ The members of the Specker sequence are bounded from below by the members of the sequence $a_n$. Since the sequence $s_n$ is strictly increasing we get that for every $n$ at least one of the inequalities $s_n<a$ and $a<s_{n+1}$ is true. If the first inequality holds for all $n$, then the Specker sequence $s_n$ has a limit
(the number $a$) which is not possible. Thus for some $m$ we have $a<s_m$ and hence $a\in A_m\subset A.$

\item The set $B$ is closed by our construction. Thus it also is sequentially closed. Indeed let us take a sequence of points $\{b_n\}_{n=1}^{\infty}$ that converges to a point $b.$ Let us show that $b\in B.$ For this is suffices to verify that $b\not\in A.$ If $b\in A$ then we can find a neighbourhood of $b$ that is contained in $A.$ This neighborhood does not contain any points of our sequence $b_n$ so $b$ can not be its limit.  

Let us show that $B$ is pseudo open. Take $b\in B$ and consider the sequence of concentric open balls $B(b,n)$. Let us show that it is impossible that there no open ball in this sequence that is contained in $B.$ Indeed assume that there is no such ball. Then every open ball $B(b,n)$ can not not contain a point of the set $A$, i.e. a point of the some $A_m$ for some $m$ and hence a member $s_m$ of the Specker sequence. The set of points of the ball $S(b,n)$ is enumerable, thus by the Markov's principle for each $n$ we can find such a member $s_m.$ Thus we can find a subsequence of the Specker sequence that converges to $b$. This is impossible so the ball in question can not not exist. Thus $B$ is pseudo open.

\end{itemize}

\end{proof}

\begin{remark}
A weak separation of a constructive interval is also its soft separtion. The separating sets $A$ and $B$ are easy to see: each of them together with any two points also contains all the points between these two points. However these sets are not constructive intervals since the point where these sets are adjacent to each other does not have to be a CRN.
\end{remark}

{\bf Acknowledgement:} The author is thankful to Andrej Bauer for posing a question that initiated this research and to my son Vladimir Chernov for motivating me to write this paper and help with the translation of the text to English.


\begin{thebibliography}{20}
\bibitem{Bauer} A. Bauer,  Private email communication, (2020)
\bibitem{BauerTaylor}A. Bauer, P. Taylor, The Dedekind Reals in Abstract Stone Duality. Mathematical Structures in Computer Science, 19 (2009)
\bibitem{Bishop} E. Bishop, D. Bridges: Constructive analysis. Grundlehren der Mathematischen Wissenschaften [Fundamental Principles of Mathematical Sciences], 279 Springer-Verlag, Berlin (1985)
\bibitem{Ceitin1} G. S. Ceitin (Tseitin): Algorithmic operators in constructive complete separable metric spaces (in Russian) Dokl. Akad. Nauk SSSR 128 (1959) 49–52. English translation in Amer. Math. Soc. Transl. 2, 64 (1967)
\bibitem{Ceitin2}  G. S. Ceitin (Tseitin): Mean-value theorems in constructive analysis (in Russian) Trudy Mat. Inst. Steklov. 67 (1962) 362–384. English translation in Amer. Math. Soc. Transl. 2, 98 (1971)
\bibitem{Chernov} V. P. Chernov: Locally Constant Constructive Functions and Connectedness of Intervals // Journal of Logic and Computations, v.30, 7, (2020) 1425-1428
\bibitem{Kleene} S. C. Kleene: Introduction to mathematics. NY-Toronto (1952)
\bibitem{Kushner} B. A. Kushner: Lectures on constructive mathematical analysis (in Russian)
Monographs in Mathematical Logic and Foundations of Mathematics. Izdat. "Nauka", Moscow, 1973. 447 pp., English translation in Translations of Mathematical Monographs, 60 American Mathematical Society, Providence, R.I. (1984). v+346 pp. ISBN: 0-8218-4513-6
\bibitem{Markov1} A. A. Markov: On constructive functions in (Russian) Trudy Mat. Inst. Steklov 52 (1958), 315–348. English translation in Amer. Math. Soc. Transl. 2, 29 (1963)
\bibitem{Markov2} A. A. Markov: On constructive mathematics (in Russian) Trudy Mat. Inst. Steklov 67 (1962), 8–14. English translation in Amer. Math. Soc. Transl. 2, 98 (1971)
\bibitem{Orevkov} V. P. Orevkov, A constructive map of the square into itself, which moves every constructive point (in Russian) Dokl. Akad. Nauk SSSR 152 (1963) 55–58. English translation in Soviet Math Dokl. 4 (1963)
\bibitem{Shanin} N. A. Sanin (Shanin), Constructive real numbers and constructive functional spaces (in Russian) Trudy Mat. Inst. Steklov 67 (1962) 15–294. English translation in Amer. Math. Soc., Providence R.I. (1968)
\bibitem{Specker} E. Specker, Nicht konstruktiv beweisbare Satze der Analysis. (German) J. Symbolic Logic 14 (1949), 145–158
\bibitem{Taylor} P. Taylor, A lambda calculus for real analysis. Journal of Logic \& Analysis 2:5 (2010) 1–-115
\bibitem{Vickers} S.. Vickers, The connected Vietoris powerlocale. Topology and its Applications 156 (11) (2009), pp. 1886--1910
\end{thebibliography}
\end{document}